\documentclass[12pt]{amsart}
\usepackage{amsmath,amsthm,amssymb}
\usepackage{mathptmx,mathrsfs,enumerate,url}
\usepackage{comment}
\usepackage{datetime}

\usepackage[backend=biber,style=numeric,sorting=nyvt]{biblatex}
\usepackage[all]{xy}

\newcommand{\dd}[1]{\[\xymatrix{#1}\]}

\usepackage{tikz-cd}

\usepackage{extpfeil}
\usepackage{mathtools}

\newtheorem{theorem}{Theorem}[section]
\newtheorem{prop}[theorem]{Proposition}
\newtheorem{corollary}[theorem]{Corollary}
\newtheorem{lem}[theorem]{Lemma}
\newtheorem{introtheorem}{Theorem}

\theoremstyle{definition}
\newtheorem{definition}[theorem]{Definition}
\newtheorem{remark}[theorem]{Remark}

\newtheorem{problem}{Problem}

\newenvironment{acknowledgements}%
	{
   \begin{center}%
	\scshape Acknowledgments
   \end{center}\vspace{6pt}
   }%

\def\bZ{{\mathbb Z}}
\def\bQ{{\mathbb Q}}

\def\Qp{{{\mathbb Q}_p}}

\def\id{{\mbox{\rm id}}}
\def\Ker{{\mbox{\rm Ker}}}
\def\Im{{\mbox{\rm Im}}}
\def\mod{{\mbox{ mod }}}

\def\Aut{{\mbox{\rm Aut}}}

\def\Out{{\mbox{\rm Out}}}
\def\Gal{{\mbox{\rm Gal}}}

\def\defeq{ \ {\stackrel{\mathrm{def}}{=}} \ }

\title[On finite subgroups of the outer automorphism groups]{Finite subgroups of outer automorphism groups of the absolute Galois groups of mixed-characteristic local fields}

\author[Yu Nishio]{Yu Nishio \\ \\ September 2026}
\address[Yu Nishio]{}
\email{nishio.math@gmail.com}

\subjclass[2020]{11S20}

\keywords{mixed-characteristic local field, absolute Galois group, anabelian geometry, mono-anabelian geometry, group of MLF-type}

\date{}

\begin{document}

\begin{abstract}
    In the present paper, we study finite subgroups of the outer automorphism groups of the absolute Galois groups of mixed-characteristic local fields from the point of view of mono-anabelian geometry. 
    We shall refer to a finite subgroup of 
    the outer automorphism group of the absolute Galois group of a mixed-characteristic local field
    as quasi-geometric if its inverse image in the automorphism group of the corresponding absolute Galois group is isomorphic to the absolute Galois group of some mixed-characteristic local field.
    It is well-known that the image of the natural injective homomorphism from the automorphism group of a mixed-characteristic local field to the outer automorphism group of the associated absolute Galois group yields a typical example of a quasi-geometric subgroup. The main result of the present paper asserts the existence of non-quasi-geometric finite subgroups of the outer automorphism groups of the absolute Galois groups of mixed-characteristic local fields.
    This result shows that the set of quasi-geometric subgroups of the outer automorphism group of the absolute Galois group of a mixed-characteristic local field cannot be characterized, in general, as the set of all finite subgroups of the outer automorphism group. 
\end{abstract} 

\maketitle

\tableofcontents

\setcounter{section}{-1}

\section*{Introduction}
In the present paper, we study finite subgroups of the outer automorphism groups of the absolute Galois groups of mixed-characteristic local fields from the point of view of mono-anabelian geometry. 

Let 
\begin{itemize}
    \item $k$ be a mixed-characteristic local field and  
    \item $\overline{k}$ an algebraic closure of $k$. 
\end{itemize}
Then we shall write 
\begin{itemize}
    \item $G_k \defeq \Gal(\overline{k}/k)$, 
    \item $\Aut(G_k)$ for the automorphism group of $G_k$, and 
    \item $\Out(G_k)$ for the outer automorphism group of $G_k$.  
\end{itemize}
We shall say that a finite subgroup of $\Out(G_k)$ is \emph{quasi-geometric} [cf. \cite{Hoshi-topics_MLF}, Definition 6.5, (ii)] 
if its inverse image in $\Aut(G_k)$ is isomorphic to the absolute Galois group of some mixed-characteristic local field. 
Let us recall that the natural homomorphism
\[
\operatorname{Aut}(k) \hookrightarrow \operatorname{Out}(G_{k})
\]
is injective [cf., e.g., \cite{Hoshi-intro_mono_MLF}, Proposition 2.1]. 
Thus, we may regard $\operatorname{Aut}(k)$ as a subgroup of $\operatorname{Out}(G_{k})$ via this injection. 
Moreover, it follows from \cite{Hoshi-topics_MLF}, Proposition 6.7, that $\operatorname{Aut}(k) \subseteq \Out(G_k)$ is a quasi-geometric subgroup. 
Furthermore, Hoshi proved that, for a Galois-specifiable mixed-characteristic local field $k$ [cf. \cite{Hoshi-topics_MLF}, Definition 6.1], the set of strictly quasi-geometric subgroups [cf. \cite{Hoshi-topics_MLF}, Definition 6.5, (iii)] 
coincides with
the $\Out(G_k)$-orbit of $\Aut(k)$ [cf. \cite{Hoshi-topics_MLF}, Theorem 6.12, (ii)]. 
On the other hand, he also gave examples in which there exist at least two $\Out(G_k)$-conjugacy classes of strictly quasi-geometric subgroups [cf. \cite{Hoshi-topics_MLF}, Remark 8.6.1, (ii)].
Under this state of affairs, one may consider the following problem:

\begin{problem}\label{problem:1}
    Is every finite subgroup of $\Out(G_k)$ quasi-geometric?
\end{problem}

We next consider the following problem for number fields: 

\begin{problem}\label{problem:2}
    Let $K$ be a number field and $\overline{K}$ an algebraic closure of $K$. 
    We shall write $G_K \defeq \Gal(\overline{K}/K)$.  
    Let $\Delta \subseteq \Out(G_K)$ be a finite subgroup. 
    Write $\widetilde{\Delta} \subseteq \Aut(G_K)$ for the inverse image of $\Delta$ by the natural surjection $\Aut(G_K) \twoheadrightarrow \Out(G_K)$. Is the profinite group $\widetilde{\Delta}$ isomorphic to the absolute Galois group of some number field?
\end{problem}

In the situation of Problem \ref{problem:2}, the Neukirch--Uchida theorem
[cf. the main theorem of \cite{Uchida_NF}] implies that the natural
homomorphism
\[
\Aut(K) \longrightarrow \Out(G_K)
\]
is an isomorphism. We regard $\Delta$ as a subgroup of $\Aut(K)$ via
the inverse of this isomorphism. 
Write
\[
L \coloneqq K^{\Delta}.
\]
Then $K/L$ is a finite Galois extension with Galois group $\Delta$.
Moreover, the natural conjugation homomorphism
\[
G_L \longrightarrow \Aut(G_K)
\]
identifies $G_L$ with the inverse image $\widetilde{\Delta}$ of
$\Delta$ in $\Aut(G_K)$. 
Consequently, $\widetilde{\Delta} \simeq G_L$. 
Thus, Problem 2 has an affirmative answer. 

The present paper gives a negative answer to Problem \ref{problem:1}.  
More precisely, the main result of the present paper proves that there exist finite subgroups of the outer automorphism groups of the absolute Galois groups of mixed-characteristic local fields that are not quasi-geometric in general. 

First, we prove the following theorem [cf. Theorem \ref{theorem:pro-pness_of_kernel}]:

\begin{introtheorem}
\label{AAAAA}
    Let $p$ be a prime number, $s$ a positive integer, $J$ a topologically finitely generated profinite group, and $P$ a closed normal subgroup of $J$ such that $P$ is a pro-$p$ group. 
    We shall write $\Aut^P(J) \subseteq \mathrm{Aut}(J)$ for the subgroup of automorphisms of $J$ that preserve $P$.  
    Let $P^{p,i,s}$ be the $i$-th term of the descending $p^s$-central series of $P$ with $P^{p,1,s}=P$. 
    Then, for each $i \ge 2$, the subgroup $\Ker(\Aut^P(J) \to \Aut(J/P^{p,i,s}))$ is a pro-$p$ group. 
\end{introtheorem}

Furthermore, we obtain the following result [cf. Theorem \ref{theorem:existence_of_torsion_Aut(G)}]:

\begin{introtheorem}
\label{BBBBB}
    Let $p$ be an odd prime number and $k$ a finite extension of $\bQ_p$. 
    Suppose that $d_k \defeq \left[ k \colon \bQ_p \right] \geq 3$. 
    Let $\mathfrak{S}_n$ be the symmetric group on $n$ letters for a positive integer $n$.
    Then, for every subgroup $H$ of $\left(\bZ/2\bZ \right)^{\lfloor \frac{d_k-1}{2} \rfloor} \rtimes \mathfrak{S}_{\lfloor \frac{d_k-1}{2} \rfloor}$ whose order is prime to $p$, there exists a subgroup of $\Aut(G_k)$ that is isomorphic to $H$.
\end{introtheorem}

Finally, we obtain the following result [cf. Corollary \ref{corollary:non-quasi_geometric_k_1}]:

\begin{introtheorem}
\label{CCCCC}
        Let $p$ be an odd prime number and $k$ a finite extension of $\bQ_p$. 
        Suppose that $\left[ k \colon \bQ_p \right] \geq 3$. 
        Then there exists a non-quasi-geometric finite subgroup of $\Out(G_k)$.
\end{introtheorem}

We note that, independently, and during the same period, Shun Ishii obtained a result similar to Theorem \ref{theorem:pro-pness_of_kernel}.

\section{Notational conventions}
\label{section0}

\subsubsection*{\sc Topological groups} 
Let $G$ be a topological group. Then we shall write $G^{\rm ab}$ for the {\it abelianization} of $G$ [i.e., the quotient of $G$ by the closure of the commutator subgroup of $G$], $G^{\textrm{ab-tor}}$ for the closure of the subgroup of torsion elements of $G^{\rm ab}$, and $G^{\rm ab/tor}$ for the quotient of $G^{\rm ab}$ by $G^{\textrm{ab-tor}}$. 
For a closed normal subgroup $P$ of $G$, 
we shall write $\Aut^P(G) \subseteq \mathrm{Aut}(G)$ for the subgroup of automorphisms of $G$ that preserve $P$ 
and $\Aut_{G/P}(G) \subseteq \mathrm{Aut}^P(G)$ for the subgroup of automorphisms of $G$ that induce the identity automorphism of the quotient $G/P$. 

Let $H$ be a profinite group. 
Note that if $H$ is topologically finitely generated, then $\Aut(H)$ has a natural profinite group structure [cf. \cite{ribes-profinite_Groups}, Proposition 4.4.3]. 
Let $p$ be a prime number and $s$ a positive integer. 
Then we shall write $H^{p,1,s} \defeq H$ and, for each positive integer $i$, $H^{p,i+1,s} \defeq \overline{(H^{p,i,s})^{p^s}\left[H,H^{p,i,s}\right]}$.  Thus, $\{H^{p,i,s}\}_{i \geq 1}$ forms the \textit{descending $p^s$-central series} of $H$. 
If, moreover, $H$ is a pro-$p$ group, and we write $\Phi(H)$ for the Frattini subgroup of $H$, 
then one verifies immediately that $H^{p,2,1} = \Phi(H)$. 

\subsubsection*{\sc Modules} 
Let $M$ be a module. If $n$ is a positive integer, then we shall write
$M[n] \subseteq M$ for the submodule obtained by forming the kernel of the endomorphism of $M$
given by multiplication by $n$.
Moreover, if $p$ is a prime number, then we shall write
\[
M[p^\infty]
\defeq
\bigcup_{n \geq 1} M[p^n].
\]

\subsubsection*{\sc Fields} 
Let $K$ be a field. Then we shall write $\mu(K)$ for the group of roots of unity in $K$. 
If, moreover, $K$ is algebraically closed and of characteristic zero, then we shall
write 
    \[
        \Lambda(K) \defeq \varprojlim_n \mu(K)[n]
    \]
--- where the projective limit is taken over the positive integers $n$ --- and refer to
$\Lambda(K)$ as the {\it cyclotome} associated to $K$. 

We shall refer to a field isomorphic to a finite extension of $\Qp$, for some prime number $p$, as a {\it mixed-characteristic local field}.  
If $k$ is a mixed-characteristic local field, then we shall write 
\begin{itemize}
 \item  $k^{(d=1)} \subseteq k$ for the [uniquely determined] minimal mixed-characteristic local field contained in $k$,
 \item $f_k$ for the absolute residue degree of $k$, 
 \item $d_k \defeq [k:k^{(d=1)}]$ for the degree of the finite extension $k/k^{(d=1)}$, and 
 \item $p_k$ for the residue characteristic of $k$. 
\end{itemize}
We shall refer to a group isomorphic to the absolute Galois group of 
a mixed-characteristic local field as a group of {\it MLF-type} 
[cf.\ \cite{Hoshi-intro_mono_MLF}, Definition 3.1].    
In the present paper, let us always regard a group of MLF-type as a 
profinite group by means of the profinite topology 
discussed in \cite{Hoshi-intro_mono_MLF}, Proposition 1.2, (i)
[cf.\ also \cite{Hoshi-intro_mono_MLF}, Proposition 1.2, (ii)].

\section{Pro-$p$ kernels in automorphism groups of profinite groups}
\label{section1}

In the present \S \ref{section1}, we prove some 
results in profinite group theory, which will be applied in \S 2.
Note that the main result of the present \S 1 may be regarded as a generalization of a well-known result in profinite group theory 
[cf. Remark \ref{remark:generalization_of_ia_pro-p} below]. 

\begin{lem}\label{lemma:profinite_structrure_of_Ker}
    Let $J$ be a topologically finitely generated profinite group, $i,s$ positive integers, $p$ a prime number, and $P$ a closed normal subgroup of $J$. 
    Then $\Aut_{J/P^{p, i,s}}(J)$ is a closed subgroup of $\Aut(J)$. 
\end{lem}

\begin{proof}
    It is immediate that, to prove that $\Aut_{J/P^{p, i,s}}(J)$  is a closed subgroup of $\Aut(J)$, it suffices to show that $\Aut^P(J)$ is a closed subgroup of $\Aut(J)$.
    Let $\{V_j\}_{j \in I}$ be a fundamental system of open neighborhoods of the identity element of $J$ consisting of open characteristic subgroups of $J$. [Note that one verifies easily that since $J$ is topologically finitely generated, such a fundamental system always exists.]
    For each $j$, write
    \[
    \psi_j \colon \Aut(J) \longrightarrow \Aut(J/V_j)
    \]
    for the natural homomorphism. 
    Since [it is immediate that] $\Aut^{PV_j/V_j}(J/V_j)$ is a closed subgroup of $\Aut(J/V_j)$, to show that $\Aut^P(J)$ is closed in $\Aut(J)$, it suffices to verify that
    \[
    \Aut^P(J) = \bigcap_j \psi_j^{-1}\bigl(\Aut^{PV_j/V_j}(J/V_j)\bigr).
    \]
    If $\alpha \in \Aut^P(J)$, then $\alpha(P) = P$.
    Hence, for every $j$, $\alpha(PV_j) = PV_j$, so that $\psi_j(\alpha) \in \Aut^{PV_j/V_j}(J/V_j)$. 
    Thus, the inclusion $\Aut^P(J) \subseteq \bigcap_j \psi_j^{-1}\bigl(\Aut^{PV_j/V_j}(J/V_j)\bigr)$ is proved.

    Next, we suppose that $\alpha \in \bigcap_j \psi_j^{-1}(\Aut^{PV_j/V_j}(J/V_j))$.  
    Then, for every $j$, we have $\alpha(PV_j) = PV_j$.  
    Let $\pi \colon J \twoheadrightarrow J/P$ be the natural projection.  
    Since $\{V_j\}_{j \in I}$ forms a fundamental system of open neighborhoods of the identity element of $J$, $\{\pi(V_j)\}_{j \in I}$ forms a fundamental system of open neighborhoods of the identity element of $J/P$.  
    Hence
    \[
    \pi\!\left(\bigcap_j PV_j\right)
     \subseteq 
     \bigcap_j \pi(PV_j)
     = \bigcap_j \pi(V_j)
     = \{1\}.
    \]
    Thus, $\bigcap_j PV_j \subseteq \Ker(\pi) = P$, i.e., $\bigcap_j PV_j = P$.  
    Consequently,
    \[
    \alpha(P)
     = \alpha\!\left(\bigcap_j PV_j\right)
     = \bigcap_j \alpha(PV_j)
     = \bigcap_j PV_j
     = P.
    \]
    Therefore, we obtain the desired equality
    \[
    \Aut^P(J) = \bigcap_j \psi_j^{-1}\bigl(\Aut^{PV_j/V_j}(J/V_j)\bigr),
    \]
    which in particular shows that $\Aut^P(J)$ is a closed subgroup of $\Aut(J)$.
\end{proof}

\begin{theorem}\label{theorem:pro-pness_of_kernel}
    Let $p$ be a prime number, $s$ a positive integer, $J$ a topologically finitely generated profinite group, and $P$ a closed normal subgroup of $J$ such that $P$ is a pro-$p$ group.
    Then, for each $i \ge 2$, the subgroup $\Aut_{J/P^{p, i,s}}(J)$ is a pro-$p$ group.
\end{theorem}

\begin{proof}
    We denote $\Aut_{J/P^{p, 2,1}}(J)$ by $K$. 
    Then it follows from Lemma \ref{lemma:profinite_structrure_of_Ker} that $K$ is a closed subgroup of $\Aut(J)$, hence has a natural profinite group structure. 
    To prove this theorem, since [it is immediate that] $\Aut_{J/P^{p, i,s}}(J) \subseteq K$, it suffices to verify that $K$ is pro-$p$, i.e., that for every finite topological quotient 
    \[f \colon K \twoheadrightarrow Q,\]
    the group $Q$ is a $p$-group.  
    Since $\Ker(f)$ is open in $K$, and $J$ is topologically finitely generated, we obtain a characteristic open subgroup $V$ of $J$ such that $\Aut_{J/V}(J) \cap K \subseteq \Ker(f)$. 
    Then we obtain a natural surjection  
    $$K/\left(K \cap \Aut_{J/V}(J)\right) \twoheadrightarrow K/\Ker(f) \simeq Q.$$ 
    Thus, to show that $Q$ is a $p$-group, it suffices to verify that the quotient $K/\left(K \cap \Aut_{J/V}(J)\right)$ is a $p$-group.

    Next, we shall write $H$ for $J/V$, $W$ for $PV/V$, and $K_V$ for the image of $K$ by the homomorphism 
    $\Aut^P(J) \to \Aut(H)$. 
    Thus, since [it is immediate that] $\Ker(K \to K_V)$ coincides with the intersection $K \cap \Aut_{H}(J)$, we have a natural isomorphism $K / \left(K \cap \Aut_{H}(J)\right) \stackrel{\sim}{\longrightarrow} K_V$.
    It follows from the definition of $K$ that $K_V$ acts trivially on $H/W$. 
    Moreover, if we write
    \[
      \rho : \operatorname{Im}\bigl(\operatorname{Aut}^{P}(J) \to \operatorname{Aut}(H)\bigr)
      \longrightarrow \operatorname{Aut}(W)
    \]
    for the homomorphism obtained by restriction to $W$, then, since $W^{p, 2,1}$ is
    characteristic in $W$, $\rho$ determines an action of $K_V$ on $W/W^{p, 2,1}$.
    On the other hand, since one has
    \[
      P/P^{p, 2,1}=\varprojlim_U (P/U)/(P/U)^{p,2,1}
    \]
    --- where the projective limit is taken over the open normal subgroups $U$ of
    $P$ [cf. \cite{ribes-profinite_Groups}, Corollary 2.8.3] --- it follows from the
    definition of $K$ that the above action of $K_V$ on $W/W^{p, 2,1}$ is trivial.
    Thus, $\rho(K_V) \subseteq \operatorname{Aut}_{W/W^{p, 2,1}}(W)$.
    Thus, since $\Aut_{W/W^{p, 2,1}}(W)$ is a $p$-group [cf. \cite{Hall-IA_p-group}, Section 1.3], $\rho(K_V)$ is a $p$-group. Thus, by considering the exact sequence 
    \[
        \xymatrix{
        1 \ar@{->}[r] & \Ker(\rho) \cap K_V \ar@{->}[r] & K_V \ar@{->}[r] & \rho(K_V) \ar@{->}[r] & 1
        ,} 
    \]
    since $\rho(K_V)$ is a $p$-group, to prove that $K_V$ is a $p$-group, it suffices to prove that $\Ker(\rho) \cap K_V$ is a $p$-group. 
    In the following, let us write $A \defeq \Ker(\rho) \cap K_V$. 
    For every element $\alpha \in A$, since $\alpha$ acts trivially on both $H/W$ and $W$, for all $h \in H$, 
    $$\alpha(h)h^{-1} \in W \cap Z_H(W) \subseteq Z(W).$$ 
    For every $\alpha \in A$, the assignment ``$h \mapsto \alpha(h)h^{-1}$'' determines a $1$-cocycle in $Z^1(H/W, Z(W))$. 
    Moreover, one verifies easily that this correspondence determines an injective homomorphism $A \hookrightarrow Z^1(H/W, Z(W))$. 
    
    Here, since $W$ is a $p$-group, it follows that $Z(W)$ is also a $p$-group, hence $Z^1(H/W, Z(W))$ is a $p$-group. Thus, one concludes that $A$ is a $p$-group, as desired. 
\end{proof}

\begin{remark}\label{remark:generalization_of_ia_pro-p}
    Theorem \ref{theorem:pro-pness_of_kernel} may be regarded as a generalization of \cite{ribes-profinite_Groups}, Lemma 4.5.5. 
   Indeed, \cite{ribes-profinite_Groups}, Lemma 4.5.5, is none other than Theorem \ref{theorem:pro-pness_of_kernel} in the case where $J = P$, $i=2$, and $s=1$.
\end{remark}

\begin{corollary}\label{corollary:equiv_of_existence_of_-l-torsion}
    In the situation of Theorem \ref{theorem:pro-pness_of_kernel}, 
    let $m$ be a positive integer that is prime to $p$ and $\Delta_m$ a finite group of order $m$. 
    Then, for each $i \geq 2$, the following two conditions are equivalent:
    \begin{itemize}
        \item[\rm (i)] There exists a subgroup of $\Aut^P(J)$ that is isomorphic to $\Delta_m$.  
        \item[\rm (ii)] 
        There exists a subgroup of $\Aut^P(J)$ whose image in $\Aut(J/P^{p, i,s})$ is isomorphic to $\Delta_m$. 
    \end{itemize}
    Moreover, for each $i > 2$, the following two conditions are equivalent: 
    \begin{itemize}
        \item[\rm (iii)] There exists a subgroup of $\Aut_{J/P}(J)$ that is isomorphic to $\Delta_m$.  
        \item[\rm (iv)] 
        There exists a subgroup of $\Aut_{J/P}(J)$ whose image in $\Aut_{J/P}(J/P^{p, i,s})$ is isomorphic to $\Delta_m$. 
    \end{itemize}
\end{corollary}
    
\begin{proof}
    The implication (i) $\Rightarrow$ (ii) follows immediately from Theorem \ref{theorem:pro-pness_of_kernel}. 
    Next, we verify the implication (ii) $\Rightarrow$ (i). 
    Let $\Delta$ be a subgroup of $\Aut^P(J)$ whose image in $\Aut(J/P^{p, i,s})$ is isomorphic to $\Delta_m$. 
    Write ${\Delta^{'}}_m \simeq \Delta_m$ for the image of $\Delta$ by the natural homomorphism $\Aut^P(J) \xrightarrow[]{} \Aut(J/P^{p, i,s})$. 
    We shall write $D_m$ for the inverse image of ${\Delta^{'}}_m$ by the natural homomorphism $\Aut^P(J) \xrightarrow[]{} \Aut(J/P^{p, i,s})$. 
    Then we have a natural exact sequence
    $$
        \xymatrix{
        1 \ar@{->}[r] & \Aut_{J/P^{p, i,s}}(J)  \ar@{->}[r] & D_m \ar@{->}[r] & {\Delta^{'}}_m \ar@{->}[r] & 1.
        }
    $$
    Thus, it follows from \cite{ribes-profinite_Groups}, Theorem 2.3.15, 
        where we take the ``$(G, K, H)$’’ of \cite{ribes-profinite_Groups}, Theorem 2.3.15,  to be $(D_m, \Aut_{J/P^{p, i,s}}(J), {\Delta^{'}}_m)$, that this exact sequence splits. This completes the proof of the first portion of this corollary. 
    Since [it is immediate that] $\Ker(\Aut_{J/P}(J) \xrightarrow[]{} \Aut_{J/P}(J/P^{p, i,s})) =  \Aut_{J/P^{p, i,s}}(J)$, 
    the implication (iii) $\Rightarrow$ (iv) follows immediately from Theorem \ref{theorem:pro-pness_of_kernel}. 
    Finally, since [it is immediate that] $\Ker(\Aut_{J/P}(J) \xrightarrow[]{} \Aut_{J/P}(J/P^{p, i,s})) =  \Aut_{J/P^{p, i,s}}(J)$, the implication (iv) $\Rightarrow$ (iii) follows from the same argument as
    that used in the proof of the implication (ii) $\Rightarrow$ (i), with $\operatorname{Aut}^{P}(J)$ and $\operatorname{Aut}(J/P^{p, i,s})$ replaced by
    $\operatorname{Aut}_{J/P}(J)$ and
    $\operatorname{Aut}_{J/P}(J/P^{p, i,s})$, respectively.
\end{proof}

\section{Finite subgroups of the automorphism groups of groups of MLF-type}
\label{section2add}

In the present \S2, we prove the existence of nontrivial finite subgroups of the automorphism groups of the absolute Galois groups of certain mixed-characteristic local fields [cf.\ Theorem~\ref{theorem:existence_of_torsion_Aut(G)} below]. 

Let $k$ be a mixed-characteristic local field and $\overline{k}$ an algebraic closure of $k$. 
We shall write $G_k \defeq \Gal(\overline{k}/k)$. 
Let $G$ be a group of MLF-type.  
Thus, by applying the various functorial group-theoretic reconstruction algorithms of \cite{Hoshi-intro_mono_MLF}, \S 3 [cf. \cite{Hoshi-intro_mono_MLF}, Definition 3.5, (i), (ii), (iii); \cite{Hoshi-intro_mono_MLF}, Definition 3.10, (ii), (iv)], to the group $G$ of MLF-type, we obtain
\begin{itemize}
 \item a prime number $p(G)$,
 \item a positive integer $f(G)$, 
 \item a positive integer $d(G)$, 
 \item a topological group $k^{\times}(G)$, 
\item a topological group $\mathcal{O}^{\prec}(G)$, 
 \item a closed normal subgroup $I(G) \subseteq G$,   
 \item a closed normal subgroup $P(G) \subseteq G$, and
 \item a topological $G$-module $\Lambda(G)$
\end{itemize}
[cf.\ also \cite{Hoshi-intro_mono_MLF}, Summary 3.15]. 
Here, let us recall [cf.\ \cite{Hoshi-intro_mono_MLF}, Proposition 3.6; \cite{Hoshi-intro_mono_MLF}, Proposition 3.11; \cite{Hoshi-intro_mono_MLF}, Proposition 4.2] 
that 

\begin{quote}
     $p(G_k)$, $f(G_k)$, $d(G_k)$, $k^{\times}(G_k)$, $\mathcal{O}^{\prec}(G_k)$, $I(G_k)$, $P(G_k)$, $\Lambda(G_k)$ coincide with
     $p_k$, $f_k$, $d_k$, the multiplicative group $k^\times$ of $k$, the group of principal units of $k$, the inertia subgroup of $G_k$, the wild inertia subgroup of $G_k$, and 
     the cyclotome $\Lambda(\overline{k})$,  
     respectively. 
\end{quote}
Moreover, let $\mathfrak{S}_n$ be the symmetric group on $n$ letters for a positive integer $n$. 

\begin{lem}\label{lemma:torsionfreeness_of_G}
    The following hold:
    \begin{itemize}
    \item[(i)] The closed normal subgroup $P(G) \subseteq G$ is a free pro-$p(G)$ group. In particular, $P(G)$ is torsion-free.
    \item[(ii)] The quotient $I(G)/P(G)$ is isomorphic to $\prod_{l \neq p(G)} \mathbb{Z}_l$.
    In particular, since each $\mathbb{Z}_l$ is torsion-free, the group $I(G)/P(G)$ is torsion-free.
    \item[(iii)] The quotient $G/I(G)$ is isomorphic to  $\widehat{\mathbb{Z}}$.
    In particular, the group $G/I(G)$ is torsion-free. 
    \item[(iv)] The group $G$ is torsion-free.
    \end{itemize}
\end{lem}
\begin{proof}
    Assertions (i), (ii), and (iii) follow immediately from \cite{Hoshi-intro_mono_MLF}, Proposition 3.6 and \cite{Hoshi-intro_mono_MLF}, Lemma 1.5, (i), (ii). 
    By considering the natural exact sequences
    \[
    1 \longrightarrow I(G) \longrightarrow G \longrightarrow G/I(G) \longrightarrow 1
    \]
    and
    \[
    1 \longrightarrow P(G) \longrightarrow I(G) \longrightarrow I(G)/P(G) \longrightarrow 1,
    \]
    assertion (iv) follows immediately from assertions (i), (ii), and (iii).
\end{proof}

\begin{definition}\label{definition:a_and_s}

\ \ 
    \begin{itemize}
        \item[(i)] We shall write $a(G) \defeq \log_{p(G)}\left(\sharp \left(k^\times(G)[p(G)^\infty]  \right)  \right)$. 
        \item[(ii)] We define $s(G)$ to be the largest integer $n \geq 1$ such that $P(G)$ acts trivially on
        \[
            \Lambda(G)/p(G)^n\Lambda(G).
        \]
    \end{itemize}
    
\end{definition}
     
\begin{lem}\label{lemma:s(G_k)}
    Suppose that $p_k$ is odd. Then the following hold:
    \begin{itemize}
        \item[(i)] The integer $a(G_k)$ is equal to the largest integer $a_k$ such that $k$ contains a primitive $p_k^{a_k}$-th root of unity.
        \item[(ii)] The integer $s(G_k)$ is equal to the largest integer $s_k$ such that the maximal tamely ramified extension of $k$ contains a primitive $p_k^{s_k}$-th root of unity.
    \end{itemize}
\end{lem}

\begin{proof}
    Assertion (i) follows immediately from Definition \ref{definition:a_and_s}, (i), \cite{Hoshi-intro_mono_MLF}, Proposition 3.6, and \cite{Hoshi-topics_MLF}, Proposition 2.5, (i). 
    Assertion (ii) follows immediately from Definition \ref{definition:a_and_s}, (ii), \cite{Hoshi-intro_mono_MLF}, Proposition 3.6, \cite{Hoshi-intro_mono_MLF}, Proposition 4.2, (iv), and the various definitions involved. 
\end{proof}

\begin{prop}[Jannsen-Wingberg]\label{proposition:gen_and_relation_G_k} 
Suppose that $d(G) \geq 3$ and that $p(G)$ is odd. 
Write $i(G) \defeq \frac{5+(-1)^{d(G)}}{2}$. 
Then there exist 
\begin{itemize}
 \item
topological generators 
$\sigma$, $\tau$, $x_0, \dots, x_{d(G)}$ of $G$ and  
 \item
positive integers $g$, $h$  

\end{itemize}
that satisfy the following four conditions:  
\begin{enumerate}
 \item[(0)] It holds that
    \[
    g\left(
    h^{p(G)-1}+h^{p(G)-2}+\cdots+h
    \right)
    \neq p(G)-1.
    \]

 \item[\rm (1)] The closed normal subgroup $P(G)$ of $G$ is pro-$p(G)$ and topologically normally generated by $x_0, \dots, x_{d(G)}$.
 \item[\rm (2)]\label{tamerelation} The equality $\sigma \tau \sigma^{-1} = \tau^{q(G)}$ holds, where we write $q(G) \defeq p(G)^{f(G)}$.
 
 \item[\rm (3)]\label{wildrelation} 
 There exists a profinite word $\delta'(S,T,X_1,\ldots,X_{i(G)-1})$ in the variables $S,T,X_1,\ldots,X_{i(G)-1}$ such that if one writes 
 $$
    \delta(S,T,X_1,\ldots,X_{d(G)}) \defeq \delta'(S, T, X_1, \cdots, X_{i(G)-1}) \left[X_{i(G)}, X_{i(G)+1}\right] \cdots \left[X_{d(G)-1}, X_{d(G)}\right], 
$$
then the following two conditions are satisfied: 
\begin{itemize}
    \item[(a)] The equality 
        \dd{
        \sigma x_0 \sigma^{-1}
        =
        (x^{h^{{p(G)}-1}}_0 \tau x^{h^{{p(G)}-2}}_0 \tau \cdots x^{h}_0 \tau)^{\frac{\pi g}{{p(G)}-1}}
        x_1^{p(G)^{s(G)}}
        \delta(\sigma, \tau, x_1, \cdots, x_{d(G)}) 
        }
        holds, where we write $\pi$ for the unique element of $\hat{\bZ} = \prod_l \bZ_l$ whose image in $\bZ_l$ is given by $1$ if $l = p(G)$ $($resp.\ by $0$ if $l \neq p(G)$$)$. 
    \item[(b)] For arbitrary elements
        $u_1,\ldots,u_{i(G)-1}\in P(G)$, the element 
        \[
          \delta'(\sigma,\tau,u_1,\ldots,u_{i(G)-1})
        \]
        belongs to the commutator subgroup of $G$. 
\end{itemize}

\end{enumerate}
Moreover, the profinite group $G$ is isomorphic to the profinite group generated by $d(G) + 3$ generators as above, subject to the above conditions. 

\end{prop}
\begin{proof}
This assertion follows from \cite{NSW-cohomology_of_nf}, Theorem 7.5.14, together with Lemma~\ref{lemma:s(G_k)}, (ii), and \cite{Hoshi-intro_mono_MLF}, Proposition 3.6. 
\end{proof}

\begin{lem}\label{lemma:linear_independent_F_p}
    Suppose that $d(G) \geq 3$ and that $p(G)$ is odd. 
    Let $\sigma, \tau, x_0, \cdots, x_{d(G)}$ be topological generators of $G$ as in Proposition \ref{proposition:gen_and_relation_G_k}. 
    For each integer $r$ such that $0 \leq r \leq d(G)$, write $\overline{x}_r$ for the image of $x_r$ in $P(G)/P(G)^{p(G),2,1}$. Then the family $\{\overline{x}_r\}_{2 \leq r \leq d(G)}$ is $\mathbb{F}_{p(G)}$-linearly independent.
\end{lem}

\begin{proof}
    For each $i \in \{0,\dots,d(G)\}$, write
    $z_i \in \mathcal{O}^{\prec}(G)^{\mathrm{ab/tor}}$ for the image of
    $x_i \in P(G)$ [cf. \cite{Hoshi-intro_mono_MLF}, Definition 3.10, (i), (ii)].
    Then it follows immediately from the proof of
    \cite{Hoshi-Nishio-outer_MLF}, Lemma 1.3, that
    $z_0,z_1,\ldots,z_{d(G)}$ generate
    $\mathcal{O}^{\prec}(G)^{\mathrm{ab/tor}}$ as a
    $\mathbb{Z}_{p(G)}$-module. 
    Here, it follows immediately from
    \cite{Hoshi-intro_mono_MLF}, Lemma 1.2, (i), (ii); 
    \cite{Hoshi-intro_mono_MLF}, Proposition 3.11, (i), that
    $\mathcal{O}^{\prec}(G)^{\mathrm{ab/tor}}$ is a free
    $\mathbb{Z}_{p(G)}$-module of rank $d(G)$. 
    Thus, if there exists a nontrivial
    $\mathbb{Z}_{p(G)}$-linear relation among
    $z_2,\ldots,z_{d(G)}$, then, by tensoring with
    $\mathbb{Q}_{p(G)}$, one would obtain a nontrivial
    $\mathbb{Q}_{p(G)}$-linear relation among
    $
    z_2\otimes 1,\ldots,z_{d(G)}\otimes 1,
    $
    which would contradict \cite{Hoshi-Nishio-outer_MLF}, Lemma 1.3. 
    Moreover, it follows from the proof of \cite{Hoshi-Nishio-outer_MLF}, Lemma 1.3, that $z_0\otimes 1$ and $z_1\otimes 1$ are $\mathbb{Q}_{p(G)}$-linearly
    dependent.  Hence the $\mathbb{Z}_{p(G)}$-submodule generated by
    $z_0$ and $z_1$ has rank at most one, so its image in
    $\mathcal{O}^{\prec}(G)^{\mathrm{ab/tor}}
    \otimes_{\mathbb{Z}_{p(G)}}\mathbb{F}_{p(G)}$
    has dimension at most one.  Since the images of
    $z_0,\ldots,z_{d(G)}$ generate this $d(G)$-dimensional
    $\mathbb{F}_{p(G)}$-vector space, the images of
    $z_2,\ldots,z_{d(G)}$ span a subspace of dimension at least
    $d(G)-1$.
    Therefore, the images of
    $x_2,\ldots,x_{d(G)}$ in
    $\mathcal{O}^{\prec}(G)^{\mathrm{ab}/\mathrm{tor}}
    \otimes_{\mathbb{Z}_{p(G)}} \mathbb{F}_{p(G)}$
    are $\mathbb{F}_{p(G)}$-linearly
    independent. 
    Since
    $\mathcal{O}^{\prec}(G)^{\mathrm{ab}/\mathrm{tor}}
    \otimes_{\mathbb{Z}_{p(G)}} \mathbb{F}_{p(G)}$
    is abelian and annihilated by $p(G)$, the natural surjection
    $P(G)\twoheadrightarrow
    \mathcal{O}^{\prec}(G)^{\mathrm{ab}/\mathrm{tor}}
    \otimes_{\mathbb{Z}_{p(G)}} \mathbb{F}_{p(G)}$
    factors through $P(G)/P(G)^{p(G),2,1}$. 
    Thus, the images of $x_2,\ldots,x_{d(G)}$ in
    $P(G)/P(G)^{p(G),2,1}$ 
    map to the $\mathbb{F}_{p(G)}$-linearly independent images of $z_2,\ldots,z_{d(G)}$ in
    $\mathcal{O}^{\prec}(G)^{\mathrm{ab}/\mathrm{tor}}
    \otimes_{\mathbb{Z}_{p(G)}} \mathbb{F}_{p(G)}$.
    In particular, the images of $x_2,\ldots,x_{d(G)}$ in
    $P(G)/P(G)^{p(G),2,1}$ are $\mathbb{F}_{p(G)}$-linearly
    independent. This completes the proof of Lemma \ref{lemma:linear_independent_F_p}.
\end{proof}

\begin{prop}\label{proposition:lift_of_Aut(G/P^3)_add} 
        Suppose that $p(G)$ is odd. Then the natural homomorphism 
        $$
        \Aut_{G/P(G)}\left(G\right)  \xrightarrow[]{} 
        \Im \left(\Aut_{G/P(G)}\left(G/P(G)^{p(G),3,s(G)}\right) \xrightarrow[]{} \Aut_{G/P(G)}\left(G/P(G)^{p(G),2,s(G)}\right) \right)
        $$
        is surjective. 
\end{prop}

\begin{proof}
This assertion follows from \cite{NSW-cohomology_of_nf}, Theorem 7.5.15, and \cite{Wingberg_dereindeutig}, Satz 2 [cf. also \cite{NSW-cohomology_of_nf}, Remark following Theorem 7.5.15], together with
\cite{Hoshi-intro_mono_MLF}, Proposition 3.6, and Lemma \ref{lemma:s(G_k)}, (ii). 
\end{proof}

\begin{theorem}\label{theorem:nontrivial_sub_Aut(G_3)}
    Suppose that $d(G) \geq 3$ and that $p(G)$ is odd. 
    Then there exists a subgroup of $\Aut_{G/P(G)}\left(G/P(G)^{p(G),3,s(G)}\right)$ that is isomorphic to $$\left(\bZ/2\bZ \right)^{\left\lfloor \frac{d(G)-1}{2}\right\rfloor} \rtimes \mathfrak{S}_{\left\lfloor \frac{d(G)-1}{2}\right\rfloor}.$$ 
    Moreover, this subgroup acts faithfully on the quotient $G/P(G)^{p(G),2,s(G)}$. 
\end{theorem}

\begin{proof}
    We shall write 
    \[
        I \coloneqq \{r\in \mathbf Z \mid i(G)\leq r\leq d(G)-1,\quad r\equiv d(G)-1 \pmod 2 \}.
    \]
    Throughout the remainder of the proof, $i$ and $j$ range over $I$.
    Moreover, let $\sigma, \tau, x_0, \dots, x_{d(G)}$ be topological generators of $G$ as in Proposition \ref{proposition:gen_and_relation_G_k}. 
    First, let $\varphi_{i,j}$ be the assignment on the topological generators of $G$ [cf. Proposition \ref{proposition:gen_and_relation_G_k}] defined by 
    \[
    x_i \mapsto x_j,\ x_{i+1} \mapsto x_{j+1},\ x_j \mapsto x_i,\ x_{j+1} \mapsto x_{i+1},
    \]
    and the identity on all the other generators of $G$. 
    Since $\varphi_{i,j}$ fixes $\sigma$ and $\tau$, $\varphi_{i,j}(\sigma)$ and $\varphi_{i,j}(\tau)$ satisfy the relation of Proposition \ref{proposition:gen_and_relation_G_k}, (2), modulo $P(G)^{p(G),3,s(G)}$. 
    Moreover, since the quotient $P(G)/P(G)^{p(G),3,s(G)}$ is nilpotent of class at most two, the images of the commutators $[x_i,x_{i+1}]$ in $P(G)^{p(G),2,s(G)}/P(G)^{p(G),3,s(G)}$ are central. 
    Thus, we have 
    \begin{align*}
    \delta(\varphi_{i,j}(\sigma), \varphi_{i,j}(\tau), \varphi_{i,j}(x_1), \dots, \varphi_{i,j}(x_{d(G)})) 
    = \delta(\sigma, \tau, \varphi_{i,j}(x_1), \dots, \varphi_{i,j}(x_{d(G)})) \\
    =\delta'(\sigma, \tau, x_1, \dots, x_{i(G)-1}) \left[\varphi_{i,j}(x_{i(G)}), \varphi_{i,j}(x_{i(G)+1})\right] \cdots \left[\varphi_{i,j}(x_{d(G)-1}), \varphi_{i,j}(x_{d(G)})\right] \\ 
    = \delta'(\sigma, \tau, x_1, \dots, x_{i(G)-1}) \left[x_{i(G)}, x_{i(G)+1}\right] \cdots \left[x_{d(G)-1}, x_{d(G)}\right] \mod P(G)^{p(G),3,s(G)}. 
    \end{align*}

    The elements
    $\varphi_{i,j}(x_0),\ldots,\varphi_{i,j}(x_{d(G)})$
    topologically normally generate the closed normal subgroup $P(G)$ of
    $G$. Moreover, $\varphi_{i,j}(\sigma)$ and $\varphi_{i,j}(\tau)$
    satisfy relation~(2) of Proposition~2.4, and
    \[
    \varphi_{i,j}(\sigma),\varphi_{i,j}(\tau),
    \varphi_{i,j}(x_0),\ldots,\varphi_{i,j}(x_{d(G)})
    \]
    satisfy relation~(3) of Proposition~2.4 modulo $P(G)^{p(G),3,s(G)}$.
    Therefore, the assignment $\varphi_{i,j}$ determines an endomorphism
    of $G/P(G)^{p(G),3,s(G)}$. 
    We shall denote this endomorphism again by $\varphi_{i,j}$. 
    Moreover, one verifies immediately that $\varphi_{i,j}^2=\id$. Thus, $\varphi_{i,j}$ is an automorphism of $G/P(G)^{p(G),3,s(G)}$. 
        
    Next, let $\varphi_{-i}$ be the assignment on topological generators of $G$ defined by 
    \[
    x_i \mapsto x_i^{-1},\ x_{i+1} \mapsto x_{i+1}^{-1},
    \]
    and the identity on all the other generators of $G$. 
    Since $\varphi_{-i}$ fixes $\sigma$ and $\tau$, $\varphi_{-i}(\sigma)$ and $\varphi_{-i}(\tau)$ satisfy the relation of Proposition \ref{proposition:gen_and_relation_G_k}, (2), modulo $P(G)^{p(G),3,s(G)}$. 
    Moreover, since [it is immediate that] $[x_i^{-1},x_{i+1}^{-1}] = [x_i,x_{i+1}]$ modulo $P(G)^{p(G),3,s(G)}$, we have 
    \begin{align*}
    \delta(\varphi_{-i}(\sigma), \varphi_{-i}(\tau), \varphi_{-i}(x_1), \dots, \varphi_{-i}(x_{d(G)})) 
    = \delta(\sigma, \tau, \varphi_{-i}(x_1), \dots, \varphi_{-i}(x_{d(G)})) \\
    =\delta'(\sigma, \tau, x_1, \dots, x_{i(G)-1}) \left[\varphi_{-i}(x_{i(G)}), \varphi_{-i}(x_{i(G)+1})\right] \cdots \left[\varphi_{-i}(x_{d(G)-1}), \varphi_{-i}(x_{d(G)})\right] \\ 
    = \delta'(\sigma, \tau, x_1, \dots, x_{i(G)-1}) \left[x_{i(G)}, x_{i(G)+1}\right] \cdots \left[x_{d(G)-1}, x_{d(G)}\right] \mod P(G)^{p(G),3,s(G)}. 
    \end{align*}

    The elements
    $\varphi_{-i}(x_0),\ldots,\varphi_{-i}(x_{d(G)})$
    topologically normally generate the closed normal subgroup $P(G)$ of
    $G$. Moreover, $\varphi_{-i}(\sigma)$ and $\varphi_{-i}(\tau)$
    satisfy relation~(2) of Proposition~2.4, and
    \[
    \varphi_{-i}(\sigma),\varphi_{-i}(\tau),
    \varphi_{-i}(x_0),\ldots,\varphi_{-i}(x_{d(G)})
    \]
    satisfy relation~(3) of Proposition~2.4 modulo $P(G)^{p(G),3,s(G)}$.
    Therefore, the assignment $\varphi_{-i}$ determines an endomorphism
    of $G/P(G)^{p(G),3,s(G)}$.
    We shall denote this endomorphism again by $\varphi_{-i}$. 
    Moreover, one verifies immediately that $\varphi_{-i}^2=\id$. Thus, $\varphi_{-i}$ is an automorphism of $G/P(G)^{p(G),3,s(G)}$. 
    Moreover, since each of $\varphi_{i,j}$ and $\varphi_{-i}$
    fixes $\sigma$ and $\tau$ and maps each $x_r$ into $P(G)$,
    each of these automorphisms induces the identity automorphism
    of $G/P(G)$. Hence
    \[
    \varphi_{i,j},\ \varphi_{-i}
    \in
    \operatorname{Aut}_{G/P(G)}
    \left(
    G/P(G)_{p(G),3,s(G)}
    \right).
    \]

    Let $\Gamma$ denote the subgroup of $\Aut_{G/P(G)}\left(G/P(G)^{p(G),3,s(G)}\right)$ generated by the automorphisms $\varphi_{i,j}$ and $\varphi_{-i}$. 
    We first prove that $\Gamma$ acts faithfully on $G/P(G)^{p(G),2,s(G)}$. 
    Since [it is immediate that] each of $\varphi_{i,j}$ and $\varphi_{-i}$ acts trivially on $G/P(G)$, $\Gamma$ acts trivially on $G/P(G)$. 
    On the other hand, since $P(G)^{p(G),2,s(G)}\subseteq P(G)^{p(G),2,1}$, there is a natural
    surjection $G/P(G)^{p(G),2,s(G)} \twoheadrightarrow G/P(G)^{p(G),2,1}$. 
    Thus, to prove that $\Gamma$ acts faithfully on $G/P(G)^{p(G),2,s(G)}$, it suffices to verify that $\Gamma$ acts faithfully on $G/P(G)^{p(G),2,1}$. 
    For each integer $r$ such that $i(G) \leq r \leq d(G)$, write $\overline{x}_r$ for the image of $x_r$ in $P(G)/P(G)^{p(G),2,1}$. 
    Then it follows from Lemma \ref{lemma:linear_independent_F_p} 
    that the elements
    $\overline{x}_r$, where $r \in \{s \in \bZ \mid i(G) \leq s \leq d(G)\}$,
    are $\mathbb{F}_{p(G)}$-linearly independent. 
    Let $\varphi$ be a nontrivial element of the subgroup generated by $\varphi_{i,j}$ and $\varphi_{-i}$. 
    Then it is immediate that there exist a permutation
    $\sigma_{\rm perm}$ of
    $\{s \in \bZ \mid i(G) \leq s \leq d(G)\}$
    and a map $\beta : \{s \in \bZ \mid i(G) \leq s \leq d(G)\} \to \{\pm 1\}$ such that 
    \[
    \varphi(\overline{x}_r)
    =
    \overline{x}_{\sigma_{\rm perm}(r)}^{\beta(r)}
    \]
    for each $r \in \{s \in \bZ \mid i(G) \leq s \leq d(G)\}$. 
    Since $\varphi$ is nontrivial and $p(G)$ is odd, it is immediate that either $\sigma_{\rm perm} \neq \mathrm{id}$ or $\beta(r)=-1$ for some $r$.
    Thus, since $p(G)$ is odd and 
    the elements $\overline{x}_r$, where $r \in \{s \in \bZ \mid i(G) \leq s \leq d(G)\}$,
    are $\mathbb{F}_{p(G)}$-linearly independent in $P(G)/P(G)^{p(G),2,1}$, it is immediate that $\varphi$ acts nontrivially on $P(G)/P(G)^{p(G),2,1}$. Therefore, $\Gamma$ acts faithfully on $G/P(G)^{p(G),2,1}$, hence also on $G/P(G)^{p(G),2,s(G)}$.

    Write $m \defeq \left\lfloor \frac{d(G)-1}{2} \right\rfloor$ and, for $a \in \left\{1, \dots, m \right\}$, write $n_a \defeq (i(G)+2(a-1), i(G)+2a-1)$. 
    Thus, since $\Gamma$ acts faithfully on $G/P(G)^{p(G),2,s(G)}$, it is immediate that $\Gamma$ may be regarded as acting on the set $\{\pm 1, \dots, \pm m\}$ by permutations and sign changes relative to the pairs $n_1, \dots, n_m$. 
    Therefore, it holds that $\Gamma$ is isomorphic to the hyperoctahedral group, i.e., that
    \[
    \Gamma \simeq 
    (\mathbb{Z}/2\mathbb{Z})^{\left\lfloor \frac{d(G)-1}{2} \right\rfloor} \rtimes \mathfrak{S}_{\left\lfloor \frac{d(G)-1}{2} \right\rfloor}.
    \]
    This completes the proof of Theorem \ref{theorem:nontrivial_sub_Aut(G_3)}. 
\end{proof}

\begin{theorem}\label{theorem:existence_of_torsion_Aut(G)}
    Suppose that $d(G) \geq 3$ and that $p(G)$ is odd. 
    Then, for every subgroup $H$ of $\left(\bZ/2\bZ \right)^{\left\lfloor \frac{d(G)-1}{2}\right\rfloor} \rtimes \mathfrak{S}_{\left\lfloor \frac{d(G)-1}{2}\right\rfloor}$ whose order is prime to $p(G)$, there exists a subgroup of $\Aut(G)$ that is isomorphic to $H$. 
\end{theorem}

\begin{proof}
    Let $H$ be a subgroup of $\left(\mathbb{Z}/2\mathbb{Z}\right)^{\left\lfloor \frac{d(G)-1}{2}\right\rfloor} \rtimes \mathfrak{S}_{\left\lfloor \frac{d(G)-1}{2}\right\rfloor}$ whose order is prime to $p(G)$.    
    Then it follows from Theorem \ref{theorem:nontrivial_sub_Aut(G_3)} that there exists a subgroup $\Gamma_3 \subseteq \Aut_{G/P(G)}\bigl(G/P(G)^{p(G),3,s(G)}\bigr)$ 
    such that 
    $\Gamma_3$ is isomorphic to $H$, 
    and every nontrivial element of $\Gamma_3$ acts nontrivially on $G/P(G)^{p(G),2,s(G)}$. 
    Let 
    $$\lambda:\Aut_{G/P(G)}\bigl(G/P(G)^{p(G),3,s(G)}\bigr)\to \Aut_{G/P(G)}\bigl(G/P(G)^{p(G),2,s(G)}\bigr)$$ 
    be the natural homomorphism. 
    In particular, 
    the restriction of
    $\lambda$ to $\Gamma_3$ is injective.
    Thus, $H_2 \defeq \lambda(\Gamma_3) \subseteq \Aut_{G/P(G)}(G/P(G)^{p(G),2,s(G)})$ is isomorphic to $H$. 
    
    Next, write
    $\rho:
    \Aut_{G/P(G)}(G)
    \to
    \Aut_{G/P(G)}\bigl(G/P(G)^{p(G),2,s(G)}\bigr)$
    for the natural homomorphism. 
    Then it follows from Proposition \ref{proposition:lift_of_Aut(G/P^3)_add} that 
    $\operatorname{Im}(\rho)=\operatorname{Im}(\lambda)$. 
    Thus, since $H_2 \subseteq \operatorname{Im}(\lambda)$, it follows that $H_2 \subseteq \operatorname{Im}(\rho)$.    
    Let $\Pi \defeq \rho^{-1}(H_2)$. Then we obtain an exact sequence
    \[
    1
    \longrightarrow
    \Aut_{G/P(G)^{p(G),2,s(G)}}(G)
    \longrightarrow
    \Pi
    \longrightarrow
    H_2
    \longrightarrow
    1.
    \]
    Since $H_2 \cong H$, the order of $H_2$ is prime to $p(G)$. 
    Thus, it follows from Theorem \ref{theorem:pro-pness_of_kernel}, 
    where we take the ``$(J, P, i, s)$’’ of Theorem \ref{theorem:pro-pness_of_kernel},  to be $(G, P(G), 2, s(G))$, that $\Aut_{G/P(G)^{p(G),2,s(G)}}(G)$ is a pro-$p(G)$ group. 
    Thus, it follows from \cite{ribes-profinite_Groups}, Theorem 2.3.15, 
    where we take the ``$(G, K, H)$’’ of \cite{ribes-profinite_Groups}, Theorem 2.3.15,  to be $(\Pi, \Aut_{G/P(G)^{p(G),2,s(G)}}(G), H_2)$, that this exact sequence splits.
    This completes the proof of Theorem \ref{theorem:existence_of_torsion_Aut(G)}.
\end{proof}

\section{Non-quasi-geometric subgroups of the outer automorphism groups of groups of MLF-type}
\label{section3}

In the present \S \ref{section3}, we prove that 
there exist non-quasi-geometric finite subgroups [cf. \ \cite{Hoshi-topics_MLF}, Definition 6.5, (ii)] of the outer automorphism groups of the absolute Galois groups of certain mixed-characteristic local fields [cf.\ Corollary~\ref{corollary:non-quasi_geometric_k_1} below].

In the present \S \ref{section3}, we maintain the notational conventions
introduced at the beginning of the preceding \S \ref{section2add}.

\begin{definition}
    We shall say that a finite subgroup of $\Out(G)$ is {\it quasi-geometric} [cf. \ \cite{Hoshi-topics_MLF}, Definition 6.5, (ii)] if its inverse image in $\Aut(G)$ is a group of MLF-type. 
\end{definition}

\begin{definition}
    We shall write
         \[
             {\rm Orb}_{\rm qg}(G)
         \]
    for the set of quasi-geometric subgroups of Out(G).
\end{definition}

\begin{theorem}\label{theorem:non-quasi-geometric_Out(G_k)}
    Let $\Gamma$ be a nontrivial finite subgroup of $\Aut(G)$ and $H$ an open subgroup of $G$ that is preserved by $\Gamma$.
    Then there exists a non-quasi-geometric finite subgroup of $\Out(H)$. 
\end{theorem}

\begin{proof}
    First, let us observe that $G$ is slim [cf. \cite{Hoshi-intro_mono_MLF}, Lemma 1.8].
    Thus, it follows from \cite{Minamide_indecomposability}, Lemma 1.6, that if the restriction of $\Gamma$ to $H$ is trivial, then $\Gamma$ acts trivially on $G$. 
    This contradicts the nontriviality of $\Gamma$. 
    We shall write $\Gamma_H$ for the [necessarily nontrivial finite] subgroup of $\Aut(H)$ determined by the natural action of $\Gamma$ on $H$ and $\overline{\Gamma}_H$ for the image of $\Gamma_H$ in $\Out(H)$. 
    On the other hand, it follows immediately from Lemma \ref{lemma:torsionfreeness_of_G}, (iv), that every group of MLF-type is torsion-free. 
    Thus, since the inverse image of $\overline{\Gamma}_H$ in $\Aut(H)$ contains $\Gamma_H$, this inverse image is never a profinite group of MLF-type. 
    Hence, $\overline{\Gamma}_H$ is not a quasi-geometric subgroup of $\Out(H)$. 
\end{proof}

\begin{corollary}
\label{corollary:non-quasi_geometric_k_1}
    Suppose that $p_k$ is odd and that $d_k \geq 3$. 
    Then there exists a non-quasi-geometric finite subgroup of $\Out(G_k)$.
\end{corollary}

\begin{proof}
    Since $d_k \geq 3$, the group
    \[
    (\mathbb Z/2\mathbb Z)^{\lfloor (d_k-1)/2\rfloor}
    \rtimes \mathfrak{S}_{\lfloor (d_k-1)/2\rfloor}
    \]
    contains a nontrivial subgroup of order prime to $p_k$. 
    Thus, this assertion follows immediately from Theorem \ref{theorem:non-quasi-geometric_Out(G_k)} in the case where $G=H=G_k$, together with 
    Theorem \ref{theorem:existence_of_torsion_Aut(G)}, the definition of a group of MLF-type [cf. \cite{Hoshi-intro_mono_MLF}, Definition 3.1], and \cite{Hoshi-intro_mono_MLF}, Proposition 3.6. 
\end{proof}

\begin{definition}
    Let us recall that the natural homomorphism
    \[
    \operatorname{Aut}(k) \hookrightarrow \operatorname{Out}(G_{k})
    \]
    from the automorphism group of the mixed-characteristic local field $k$ to the outer automorphism group of the absolute Galois group $G_{k}$ of $k$ is injective [cf. \cite{Hoshi-intro_mono_MLF}, Proposition 2.1]. 
    We shall write 
    \[
        {\rm Orb}^{\rm all}(\Aut(k))
    \]
    for the set of $\Out(G_k)$-conjugates of subgroups of $\Aut(k)$.  
\end{definition}

\begin{corollary}
\label{corollary:non-quasi_geometric_k_3}
    Suppose that $p_k$ is odd and that $d_k \geq 3$. 
    Then there exists a finite subgroup of $\Out(G_k)$ that is not 
    in ${\rm Orb}^{\rm all}(\Aut(k))$. 
\end{corollary}

\begin{proof}
    It follows from \cite{Hoshi-topics_MLF}, Proposition 6.7, that $\Aut(k) \subseteq \Out(G_k)$ is quasi-geometric. 
    Moreover, every subgroup of a quasi-geometric subgroup is quasi-geometric [cf. \cite{Hoshi-topics_MLF}, Lemma 6.6, (i)]. Since any $\Out(G_k)$-conjugate of a quasi-geometric subgroup is quasi-geometric by definition, every member of ${\rm Orb}^{\rm all}(\Aut(k))$ is a quasi-geometric subgroup of $\Out(G_k)$. 
    Thus, since ${\rm Orb}^{\rm all}(\Aut(k)) \subseteq {\rm Orb}_{\rm qg}(G_k)$, Corollary \ref{corollary:non-quasi_geometric_k_3} follows immediately from Corollary \ref{corollary:non-quasi_geometric_k_1}.  
\end{proof}

\begin{acknowledgements}
The author would like to express deep gratitude to Yuichiro Hoshi for many discussions related to this paper, and would also like to thank Shun Ishii for valuable comments and discussions. 
Finally, the author thanks K. Nishio for constant support and warm encouragement.
\end{acknowledgements}

\printbibliography

\end{document}